\documentclass[reqno]{amsart}

\usepackage{graphicx}
\usepackage{amscd}
\usepackage{amsmath, amssymb}
\usepackage[all,knot]{xy}
\usepackage{tikz}
\usepackage[T1]{fontenc}
\usepackage{lmodern}
\usepackage{morefloats}
\usepackage{tikz}
\usepackage{braids}
\usetikzlibrary{braids}
\usepackage[numbered]{hypdoc}
\definecolor{lstbgcolor}{rgb}{0.9,0.9,0.9} 
 
\usepackage{listings}
\usepackage{fancyvrb}

\newtheorem{theorem}{Theorem}

\newtheorem{definition}[theorem]{Definition}

\begin{document}

\title[THE FAITHFULNESS OF AN EXTENSION OF LAWRENCE-KRAMMER REPRESENTATION]{THE FAITHFULNESS OF AN EXTENSION OF LAWRENCE-KRAMMER REPRESENTATION ON THE GROUP OF CONJUGATING AUTOMORPHISMS $C_n$ in the cases $n=3$ and $n=4$}

\author{Mohamad N. Nasser}

\address{Mohamad N. Nasser\\
         Department of Mathematics\\
         Beirut Arab University\\
         P.O. Box 11-5020, Beirut, Lebanon}
         
\email{m.nasser@bau.edu.lb}

\maketitle

\begin{abstract}
Let $C_n$ be the group of conjugating automorphisms. Valerij G. Bardakov defined a representation $\rho$ of $C_n$, which is an extension of Lawrence-Krammer representation of the braid group $B_n$. Bardakov proved that the representation $\rho$ is unfaithful for $n \geq 5$. The cases $n=3,4$ remain open. M. N. Nasser and M. N. Abdulrahim made attempts towards the faithfulness of $\rho$ in the case $n=3$. In this work, we prove that $\rho$ is unfaithful in the both cases $n=3$ and $n=4$.
\end{abstract}

\medskip

\renewcommand{\thefootnote}{}

\vskip 0.1in

\section{Introduction} 

The braid group on $n$ strings, $B_n$, is the abstract group with generators $\sigma_1,\ldots,\sigma_{n-1}$ and a presentation as follows:
\begin{align*}
&\sigma_i\sigma_{i+1}\sigma_i = \sigma_{i+1}\sigma_i\sigma_{i+1} ,\hspace{0.5cm} i=1,2,\ldots,n-2,\\
&\sigma_i\sigma_j = \sigma_j\sigma_i , \hspace{2.1cm} |i-j|\geq 2.
\end{align*}

Let $\mathbb{F}_n$ be a free group of $n$ generators $x_1,x_2,\ldots,x_n$. The group of conjugating automorphisms, $C_n$, is one of the generalizations of the braid group $B_n$ [1]. The group $C_n$ is defined to be the subgroup of $Aut(\mathbb{F}_n)$ that satisfies for any $\Phi \in C_n$, $\Phi(x_i)={f_i}^{-1}x_{\Pi(i)}f_i$, where $\Pi$ is a permutation on $\{1,2,\ldots,n\}$ and $f_i=f_i(x_1,x_2,\ldots,x_n)$.

Lawrence-Krammer representation is one of the most famous linear representations of the braid group $B_n$ [4]. Braid groups are linear due to Lawrence-Krammer representations. It was shown that Lawrence-Krammer representations are faithful for all $n\in \mathbb{N}$ [2]. In [1], Bardakov uses Magnus representation defined in [3] to construct a linear representation $\rho: C_n \mapsto GL(V_n)$, where $V_n$ is a free module of dimension $n(n-1)/2$ with a basis $\{v_{i,j}\}, 1\leq i<j\leq n$. This representation is shown to be an extension of Lawrence-Krammer representation of $B_n$.

Valerij G. Bardakov showed that the representation $\rho$ is unfaithful for $n\geq 5$ [1]. Moreover, M. N. Nasser and M. N. Abdulrahim proved that $\rho$ is unfaithful under some choices of $q$ in the case $n=3$ [5]. In addition, they found the shape of the all possible elements in $\ker\rho$ when $q^{6k}\neq 1$ for all $k\in \mathbb{Z}$. However, the question of faithfulness of $\rho$ is still open in the both cases $n=3$ and $n=4$. 

In section 3 of our work, we prove that the representation $\rho$ is unfaithful in the both cases $n=3$ and $n=4$ (Theorem 3 and Theorem 5).

\section{Preliminaries} 

The group of conjugating automorphisms, $C_n$, is the subgroup of $Aut(\mathbb{F}_n)$ that satisfies for any $\Phi \in C_n$, $\Phi(x_i)={f_i}^{-1}x_{\Pi(i)}f_i$, where $\Pi$ is a permutation on $\{1,2,\ldots ,n\}$ and $f_i=f_i(x_1,x_2,\ldots ,x_n)$. Here $\mathbb{F}_n$ is the free group of $n$ generators $x_1,x_2,\ldots ,x_n$.

A. G. Savushkina [6] proved that $C_n$ is generated by automorphisms $\sigma_1,\sigma_2,\ldots ,\sigma_{n-1},\\ \alpha_1,\alpha_2,\ldots ,\alpha_{n-1}$ of the free group $\mathbb{F}_n$, where $\sigma_1,\sigma_2,\ldots ,\sigma_{n-1}$ generate the braid group $B_n$, and $\alpha_1,\alpha_2,\ldots ,\alpha_{n-1}$ generate the symmetric group $S_n$.

In [1], we see that the group $C_n$ is defined by the relations:
\begin{align*}
\sigma_i\sigma_{i+1}\sigma_i = \sigma_{i+1}\sigma_i\sigma_{i+1} ,\hspace{.3cm} &\text{for} \hspace{.3cm} i=1,2,\ldots ,n-2,\vspace{0.1cm}\\ 
\sigma_i\sigma_j = \sigma_j\sigma_i , \hspace{.3cm} &\text{for} \hspace{.3cm} |i-j|>2, \vspace{0.1cm}\\
\alpha^{2}_i=1, \hspace{.3cm} &\text{for} \hspace{.3cm} i=1,2,\ldots ,n-1,  \vspace{0.1cm}\\
\alpha_j\alpha_{j+1}\alpha_j=\alpha_{j+1}\alpha_{j}\alpha_{j+1}, \hspace{.3cm} &\text{for} \hspace{.3cm} j=1,2,\ldots ,n-2,  \vspace{0.1cm}\\
\alpha_i\alpha_j=\alpha_j\alpha_i, \hspace{.3cm} &\text{for} \hspace{.3cm} |i-j|\geq 2,  \vspace{0.1cm}\\
\alpha_i\sigma_j=\sigma_j\alpha_i, \hspace{.3cm} &\text{for} \hspace{.3cm} |i-j|\geq 2, \vspace{0.1cm}\\
\sigma_i\alpha_{i+1}\alpha_i=\alpha_{i+1}\alpha_i\sigma_{i+1}, \hspace{.3cm} &\text{for} \hspace{.3cm} i=1,2,\ldots ,n-2, \vspace{0.1cm}\\
\sigma_{i+1}\sigma_{i}\alpha_{i+1}=\alpha_{i}\sigma_{i+1}\sigma_{i}, \hspace{.3cm} &\text{for} \hspace{.3cm} i=1,2,....,n-2.
\end{align*}

\begin{definition}
\text{[4]}
Let $V_n$ be a free module of dimension $n(n-1)/2$ and a basis $\{v_{i,j}\}, 1\leq i< j\leq n$ over the ring $\mathbb{Z}[q^{\pm1}]$ of Laurent polynomials in one variable. We introduce the representation $\rho:C_n \mapsto GL(V_n)$ by the actions of $\sigma_i's$ and $\alpha_i's$, $i=1, \ldots, n-1$ on the basis of the module $V_n$ as follows:
\begin{align*}
\left\{\begin{array}{l}
\sigma_i(v_{k,i})=(1-q)v_{k,i}+qv_{k,{i+1}}+q(q-1)v_{i,i+1},\\
\sigma_i(v_{k,{i+1}})=v_{k,i}$, \hspace{0.5cm} $k<i,\\
\sigma_i(v_{i,{i+1}})=q^2v_{i,{i+1}},\\
\sigma_i(v_{i,l})=q(q-1)v_{i,{i+1}}+(1-q)v_{i,l}+qv_{i+1,l},$ \hspace{0.5cm}$ i+1<l,\\
\sigma_i(v_{i+1,l})=v_{i,l},\\
\sigma_i(v_{k,l})=v_{k,l}$, \hspace{0.5cm} $\{ k,l\} \cap \{i,i+1\}=\emptyset, \\
\alpha_i(v_{k,i})=v_{k,{i+1}},\\
\alpha_i(v_{k,{i+1}})=v_{k,i}$, \hspace{0.5cm} $k<i,\\
\alpha_i(v_{i,{i+1}})=v_{i,{i+1}},\\
\alpha_i(v_{i,l})=v_{i+1,l},$ \hspace{0.5cm}$ i+1<l,\\
\alpha_i(v_{i+1,l})=v_{i,l},\\
\alpha_i(v_{k,l})=v_{k,l}$, \hspace{0.5cm} $\{ k,l\} \cap \{i,i+1\}=\emptyset. \\
\end{array}\right.
\end{align*}
\end{definition}
\vspace{0.2cm}

\section{The representation $\rho$ is unfaithful for $n=3$ and $n=4$} 

It was proven that Lawrence-Krammer representation of $B_n$ is faithful for all $n \in \mathbb{N}$ [2]. The representation $\rho:C_n \mapsto GL_{\frac{n(n-1)}{2}}(\mathbb{Z}[q^{\pm1}])$ is an extension to $C_n$ of Lawrence-Krammer representations of $B_n$ [1]. Bardakov proved that $\rho$ is unfaithful for $n\geq 5$ [1]. M. N. Nasser and M. N. Abdulrahim proved, under special choices of $q$, that $\rho$ is unfaithful in the case $n=3$ [5]. Moreover, they proved that if $q^{6k}\neq 1$ for all $k \in \mathbb{Z}$ then the possible words in $\ker \rho$ are $A_1T^{s_1}A_2T^{s_2}\ldots A_{r-1}T^{s_{r-1}}A_rT^{s_r}$ and $T^{s_1}A_1T^{s_2}A_2\ldots T^{s_{r-1}}A_{r-1}T^{s_r}A_r$, where $T=\sigma_2\alpha_2\alpha_1, r\in \mathbb{N}, s_i \in \mathbb{Z}$ for all $1\leq i \leq r, \displaystyle \sum_{i=1}^{r}s_i =0, \displaystyle \sum_{i=1}^{r} length (A_i)$ is even and $A_i \in \{\alpha_1, \alpha_2, \alpha_1\alpha_2, \alpha_2\alpha_1, \alpha_1\alpha_2\alpha_1\}$ for all $1 \leq i \leq r$. The question of faithfulness of the representation $\rho$ is still open for $n=3$ and $n=4$. We answer the question for the complex specialization of the representation $\rho$ in the both case $n=3$ and $n=4$ by showing that $\rho$ is unfaithful.

\vspace{.5cm}
In what follows we consider the representation $\rho$ for $n=3$.

\begin{definition}
Consider the representation $\;\rho:C_3 \mapsto GL_3(\mathbb{Z}[q^{\pm1}])$ and specialize $q$ to a non zero complex number. The complex specialization of $\rho $ is defined by the actions of $\sigma_1$, $\sigma_2$, $\alpha_1$ and $\alpha_2$ on the standard unit vectors  $\{e_1,e_2,e_3\}$ of $\mathbb C^3$ as follows:\vspace{0.3cm} \\
$$\sigma_1 \mapsto \left\{\begin{array}{l}
e_1\mapsto q^2e_1\\
e_2\mapsto q(q-1)e_1+(1-q)e_2+qe_3\\
e_3\mapsto e_2
\end{array}\right., \hspace{0.2cm} \sigma_2 \mapsto \left\{\begin{array}{l}
e_1\mapsto (1-q)e_1+qe_2+q(q-1)e_3\\
e_2\mapsto e_1\\
e_3\mapsto q^2e_3
\end{array}\right.,$$
\\
$$\alpha_1 \mapsto \left\{\begin{array}{l}
e_1\mapsto e_1\\
e_2\mapsto e_3\\
e_3\mapsto e_2
\end{array}\right.
\hspace{0.2cm} \text{and} \hspace{0.3cm} \alpha_2 \mapsto \left\{\begin{array}{l}
e_1\mapsto e_2\\
e_2\mapsto e_1\\
e_3\mapsto e_3
\end{array}\right..$$

\vspace{0.5cm}
In other words, for $n=3$, the representation $\rho$ is given by 
$$\rho:C_3 \mapsto GL(\mathbb{C}^3)$$
$$\rho(\sigma_1)=
\left( \begin{array}{@{}c@{}}
\begin{matrix}
   		q^2 & q(q-1) & 0 \\
    	0 & 1-q & 1 \\
        0 & q & 0 \\
\end{matrix}
\end{array} \right), \hspace{0.5cm}
\rho(\sigma_2)=
\left( \begin{array}{@{}c@{}}
\begin{matrix}
   		1-q & 1 & 0 \\
    	q & 0 & 0 \\
        q(q-1) & 0 & q^2 \\
\end{matrix}
\end{array} \right),$$
\\
$$\rho(\alpha_1)=
\left( \begin{array}{@{}c@{}}
\begin{matrix}
   		1 & 0 & 0 \\
    	0 & 0 & 1 \\
        0 & 1 & 0 \\
\end{matrix}
\end{array} \right) \hspace{0.2cm}\text{and} \hspace{0.5cm}
\rho(\alpha_2)=
\left( \begin{array}{@{}c@{}}
\begin{matrix}
   		0 & 1 & 0 \\
    	1 & 0 & 0 \\
        0 & 0 & 1 \\
\end{matrix}
\end{array} \right).$$
\end{definition}

\vspace*{0.25cm}

Notice that $q\neq 1$ since otherwise we get $\alpha_1=\sigma_1$ and $\alpha_2=\sigma_2$.

\vspace{0.5cm}

We now prove that $\rho$ is unfaithful in the case $n=3$ by specifying an element in its kernel.

\begin{theorem}
The complex specialization of the representation $\;\rho:C_3 \mapsto GL_3(\mathbb{Z}[q^{\pm1}])$ is unfaithful.
\end{theorem}

\begin{proof}
Fix $q\in \mathbb{C}^*$ and consider the word $v=\sigma_2^{-1}\alpha_2\alpha_1\sigma_2^{-1}\alpha_2\alpha_1\sigma_2\alpha_1\alpha_2\sigma_2\alpha_1\alpha_2$. Suppose that $v$ is a trivial word, that is $v=id_{C_3}$. Then, we have\\
$v=id_{C_3}\iff \sigma_2^{-1}\alpha_2\alpha_1\sigma_2^{-1}\alpha_2\alpha_1\sigma_2\alpha_1\alpha_2\sigma_2\alpha_1\alpha_2=id_{C_3}\\
\hspace*{1.3cm}\iff \sigma_2\alpha_1\alpha_2\sigma_2\alpha_1\alpha_2=(\sigma_2^{-1}\alpha_2\alpha_1\sigma_2^{-1}\alpha_2\alpha_1)^{-1}\\
\hspace*{1.3cm}\iff \sigma_2\alpha_1\alpha_2\sigma_2\alpha_1\alpha_2=\alpha_1\alpha_2\sigma_2\alpha_1\alpha_2\sigma_2 \hspace{1.5cm} (\alpha_1^{-1}=\alpha_1, \alpha_2^{-1}=\alpha_2)\\
\hspace*{1.3cm}\iff \alpha_1\alpha_2\sigma_1 \sigma_2\alpha_1\alpha_2=\alpha_1\alpha_2\sigma_2\alpha_1\alpha_2\sigma_2 \hspace{1.5cm} (\alpha_1\alpha_2\sigma_1=\sigma_2\alpha_1\alpha_2)\\
\hspace*{1.3cm}\iff \sigma_1 \sigma_2\alpha_1\alpha_2=\sigma_2\alpha_1\alpha_2\sigma_2 \\
\hspace*{1.3cm}\iff \sigma_1 \sigma_2\alpha_1\alpha_2=\alpha_1\alpha_2\sigma_1\sigma_2 \hspace{3cm} (\alpha_1\alpha_2\sigma_1=\sigma_2\alpha_1\alpha_2)\\
\hspace*{1.3cm}\iff$ The geometrical shapes of $\sigma_1\sigma_2\alpha_1\alpha_2$ and $\alpha_1\alpha_2\sigma_1\sigma_2$ are the same\\
\hspace*{2.2cm} (look Figure 1).
\begin{figure}[h]
\centering
\includegraphics[scale=0.5]{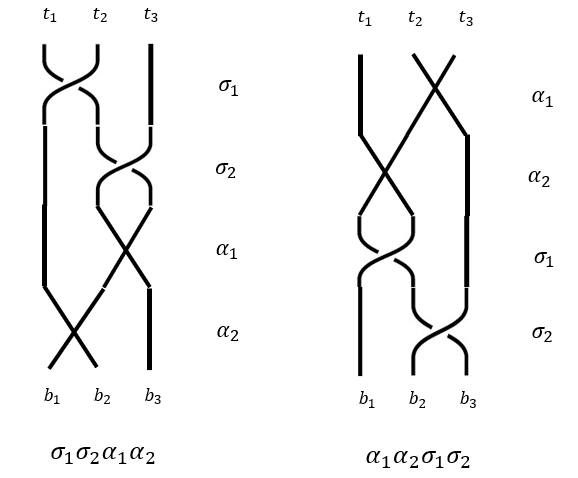}
\caption{}
\end{figure}

But we can see that the geometrical shapes of $\sigma_1\sigma_2\alpha_1\alpha_2$ and $\alpha_1\alpha_2\sigma_1\sigma_2$ in Figure 1 are not the same. Indeed, In the left hand side braid, $\sigma_1\sigma_2\alpha_1\alpha_2$, we see that the string that connect the bottom node $b_1$ to the top node $t_1$ is below the string that connect the bottom node $b_3$ to the top node $t_3$. While in the right hand side braid, $\alpha_1\alpha_2\sigma_1\sigma_2$, the string that connect the bottom node $b_1$ to the top node $t_1$ is above the string that connect the bottom node $b_3$ to the top node $t_3$. Therefore, the two words $\sigma_1\sigma_2\alpha_1\alpha_2$ and $\alpha_1\alpha_2\sigma_1\sigma_2$ have different geometrical shapes and so $\sigma_1 \sigma_2\alpha_1\alpha_2\neq \alpha_1\alpha_2\sigma_1\sigma_2$, which is a contradiction. Thus, $v$ is not trivial word.

Now, direct computations show that $\rho(v)=\rho(\sigma_2^{-1}\alpha_2\alpha_1\sigma_2^{-1}\alpha_2\alpha_1\sigma_2\alpha_1\alpha_2\sigma_2\alpha_1\alpha_2)=I_3$. This implies that $v\in \ker(\rho)$ and so $\rho$ is unfaithful.
\end{proof}

Now, we consider the representation $\rho$ for $n=4$.

\begin{definition}
Consider the representation $\;\rho:C_4 \mapsto GL_6(\mathbb{Z}[q^{\pm1}])$ and specialize $q$ to a non zero complex number. The complex specialization of $\rho $ is defined by the actions of $\sigma_1$, $\sigma_2$, $\sigma_3$, $\alpha_1$, $\alpha_2$ and $\alpha_3$ on the standard unit vectors  $\{e_1,e_2,e_3,e_4,e_5,e_6\}$ of $\mathbb C^6$ as follows:\vspace{0.3cm} \\
$$\sigma_1 \mapsto \left\{\begin{array}{l}
e_1\mapsto q^2e_1\\
e_2\mapsto q(q-1)e_1+(1-q)e_2+qe_4\\
e_3\mapsto q(q-1)e_1+(1-q)e_3+qe_5\\
e_4 \mapsto e_2\\
e_5 \mapsto e_3\\
e_6 \mapsto e_6
\end{array}\right., \hspace{0.4cm} \sigma_2 \mapsto \left\{\begin{array}{l}
e_1\mapsto (1-q)e_1+qe_2+q(q-1)e_4\\
e_2\mapsto e_1\\
e_3\mapsto e_3\\
e_4\mapsto q^2e_4\\
e_5\mapsto q(q-1)e_4+(1-q)e_5+qe_6\\
e_6\mapsto e_5
\end{array}\right.,$$
\\
$$\sigma_3 \mapsto \left\{\begin{array}{l}
e_1\mapsto e_1\\
e_2\mapsto (1-q)e_2+qe_3+q(q-1)e_6\\
e_3\mapsto e_2\\
e_4\mapsto (1-q)e_4+qe_5+q(q-1)e_6\\
e_5\mapsto e_4\\
e_6\mapsto q^2e_6
\end{array}\right., \hspace{0.4cm} \alpha_1 \mapsto \left\{\begin{array}{l}
e_1\mapsto e_1\\
e_2\mapsto e_4\\
e_3\mapsto e_5\\
e_4\mapsto e_2\\
e_5\mapsto e_3\\
e_6 \mapsto e_6
\end{array}\right.$$
\\
$$
\alpha_2 \mapsto \left\{\begin{array}{l}
e_1\mapsto e_2\\
e_2\mapsto e_1\\
e_3\mapsto e_3\\
e_4\mapsto e_4\\
e_5 \mapsto e_6\\
e_6 \mapsto e_5
\end{array}\right.,\hspace{0.4cm} \text{and} \hspace{0.4cm} \alpha_3 \mapsto \left\{\begin{array}{l}
e_1\mapsto e_1\\
e_2\mapsto e_3\\
e_3\mapsto e_2\\
e_4\mapsto e_5\\
e_5 \mapsto e_4\\
e_6 \mapsto e_6
\end{array}\right..$$

\vspace{0.5cm}
In other words, for $n=4$, the representation $\rho$ is given by 
$$\rho:C_4 \mapsto GL(\mathbb{C}^6)$$
$$\rho(\sigma_1)=
\left( \begin{array}{@{}c@{}}
\begin{matrix}
   		q^2 & q(q-1) & q(q-1) & 0 & 0 & 0\\
    	0 & 1-q & 0 & 1 & 0 & 0 \\
        0 & 0 & 1-q & 0 & 1 & 0 \\
        0 & q & 0 & 0 & 0 & 0 \\
        0 & 0 & q & 0 & 0 & 0 \\
        0 & 0 & 0 & 0 & 0 & 1 \\
\end{matrix}
\end{array} \right), \hspace{0.5cm}
\rho(\sigma_2)=
\left( \begin{array}{@{}c@{}}
\begin{matrix}
   		1-q & 1 & 0 & 0 & 0 & 0 \\
    	q & 0 & 0 & 0 & 0 & 0 \\
        0 & 0 & 1 & 0 & 0 & 0 \\
        q(q-1) & 0 & 0 & q^2 & q(q-1) & 0 \\
        0 & 0 & 0 & 0 & 1-q & 1 \\
        0 & 0 & 0 & 0 & q & 0 \\
\end{matrix}
\end{array} \right),$$

$$\rho(\sigma_3)=
\left( \begin{array}{@{}c@{}}
\begin{matrix}
   		1 & 0 & 0 & 0 & 0 & 0\\
    	0 & 1-q & 1 & 0 & 0 & 0 \\
        0 & q & 0 & 0 & 0 & 0 \\
        0 & 0 & 0 & 1-q & 1 & 0 \\
        0 & 0 & 0 & q & 0 & 0 \\
        0 & q(q-1) & 0 & q(q-1) & 0 & q^2 \\
\end{matrix}
\end{array} \right), \hspace{0.5cm}
\rho(\alpha_1)=
\left( \begin{array}{@{}c@{}}
\begin{matrix}
   		1 & 0 & 0 & 0 & 0 & 0 \\
    	0 & 0 & 0 & 1 & 0 & 0 \\
        0 & 0 & 0 & 0 & 1 & 0 \\
        0 & 1 & 0 & 0 & 0 & 0 \\
        0 & 0 & 1 & 0 & 0 & 0 \\
        0 & 0 & 0 & 0 & 0 & 1 \\
\end{matrix}
\end{array} \right),$$
\\
$$\rho(\alpha_2)=
\left( \begin{array}{@{}c@{}}
\begin{matrix}
   		0 & 1 & 0 & 0 & 0 & 0 \\
    	1 & 0 & 0 & 0 & 0 & 0 \\
        0 & 0 & 1 & 0 & 0 & 0 \\
        0 & 0 & 0 & 1 & 0 & 0 \\
        0 & 0 & 0 & 0 & 0 & 1 \\
        0 & 0 & 0 & 0 & 1 & 0 \\
\end{matrix}
\end{array} \right) \hspace{0.5cm}\text{and} \hspace{0.5cm}
\rho(\alpha_3)=
\left( \begin{array}{@{}c@{}}
\begin{matrix}
   		1 & 0 & 0 & 0 & 0 & 0 \\
    	0 & 0 & 1 & 0 & 0 & 0 \\
        0 & 1 & 0 & 0 & 0 & 0 \\
        0 & 0 & 0 & 0 & 1 & 0 \\
        0 & 0 & 0 & 1 & 0 & 0 \\
        0 & 0 & 0 & 0 & 0 & 1 \\
\end{matrix}
\end{array} \right).$$
\end{definition}

\vspace*{0.5cm}

Notice that $q\neq 1$ since otherwise we get $\alpha_1=\sigma_1$, $\alpha_2=\sigma_2$ and $\alpha_3=\sigma_3$.
\vspace*{0.5cm}

We now prove that $\rho$ is unfaithful in the case $n=4$ by specifying an element in its kernel.

\begin{theorem}
The complex specialization of the representation $\;\rho:C_4 \mapsto GL_6(\mathbb{Z}[q^{\pm1}])$ is unfaithful.
\end{theorem}
\begin{proof}
Fix $q\in \mathbb{C}^*$ and consider the word $w=\sigma_1\alpha_1\alpha_2\alpha_1\sigma_1^{-1}\alpha_2\alpha_1\sigma_1^{-1}\alpha_2\sigma_1\alpha_1\alpha_2$. Suppose that $w$ is a trivial word, that is $w=id_{C_4}$. Then, we have\\
$w=id_{C_4}\iff \sigma_1\alpha_1\alpha_2\alpha_1\sigma_1^{-1}\alpha_2\alpha_1\sigma_1^{-1}\alpha_2\sigma_1\alpha_1\alpha_2=id_{C_4}\\
\hspace*{1.37cm}\iff \sigma_1\alpha_2\alpha_1\alpha_2\sigma_1^{-1}\alpha_2\alpha_1\sigma_1^{-1}\alpha_2\sigma_1\alpha_1\alpha_2=id_{C_4} \hspace{0.78cm} (\alpha_1\alpha_2\alpha_1=\alpha_2\alpha_1\alpha_2)\\ \hspace*{1.37cm}\iff \alpha_2\alpha_1\sigma_2\alpha_2\sigma_1^{-1}\alpha_2\alpha_1\sigma_1^{-1}\alpha_2\sigma_1\alpha_1\alpha_2=id_{C_4} \hspace{0.78cm} (\sigma_1\alpha_1\alpha_2=\alpha_2\alpha_1\sigma_2)\\ \hspace*{1.37cm}\iff \sigma_2\alpha_2\sigma_1^{-1}\alpha_2\alpha_1\sigma_1^{-1}\alpha_2\sigma_1=id_{C_4} \hspace{2.3cm} (\alpha_1^{-1}=\alpha_1, \alpha_2^{-1}=\alpha_2)\\ \hspace*{1.37cm}\iff \sigma_2\alpha_2\sigma_1^{-1}\alpha_2\alpha_1=(\sigma_1^{-1}\alpha_2\sigma_1)^{-1}\\ \hspace*{1.37cm}\iff \sigma_2\alpha_2\sigma_1^{-1}\alpha_2\alpha_1=\sigma_1^{-1}\alpha_2\sigma_1 \hspace{2.93cm} (\alpha_2^{-1}=\alpha_2)\\ \hspace*{1.37cm}\iff \sigma_1\sigma_2\alpha_2\sigma_1^{-1}\alpha_2\alpha_1=\alpha_2\sigma_1\\ \hspace*{1.37cm}\iff \sigma_1\sigma_2\alpha_2=\alpha_2\sigma_1\alpha_1\alpha_2\sigma_1 \hspace{3.35cm} (\alpha_1^{-1}=\alpha_1,\alpha_2^{-1}=\alpha_2)\\ \hspace*{1.37cm}\iff \sigma_1\sigma_2\alpha_2\alpha_2\alpha_1=\alpha_2\sigma_1\alpha_1\alpha_2\sigma_1\alpha_2\alpha_1\\ \hspace*{1.37cm}\iff \sigma_1\sigma_2\alpha_1=\alpha_2\sigma_1\sigma_2 \hspace{4.1cm} (\alpha_1\alpha_2\sigma_1\alpha_2\alpha_1=\sigma_2)\\ \hspace*{1.37cm}\iff
$ The geometrical shapes of $\sigma_1\sigma_2\alpha_1$ and $\alpha_2\sigma_1\sigma_2$ are the same\\
\hspace*{2.2cm} (look Figure 2).
\begin{figure}[h]
\centering
\includegraphics[scale=0.5]{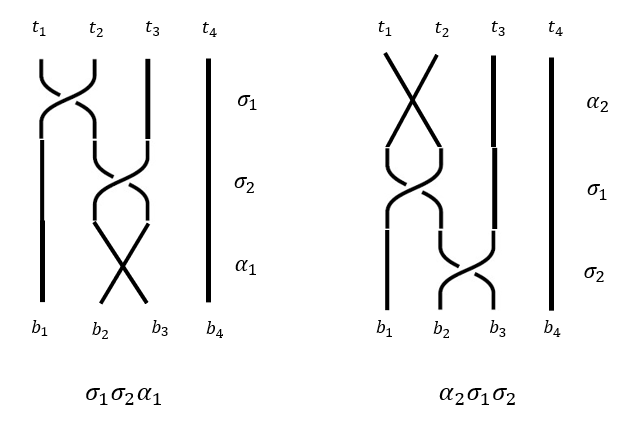}
\caption{}
\end{figure}

But we can see that the geometrical shapes of $\sigma_1\sigma_2\alpha_1$ and $\alpha_2\sigma_1\sigma_2$ in Figure 2 are not the same, since the bottom node $b_1$ in the left hand side braid, $\sigma_1\sigma_2\alpha_1$, has 1 string connected to the top node $t_2$ and no strings connected to the node $t_1$, while in the right hand side braid, $\alpha_2\sigma_1\sigma_2$, the bottom node $b_1$ has two strings connected to the top nodes $t_1$ and $t_2$ each. Therefore, the two words $\sigma_1\sigma_2\alpha_1$ and $\alpha_2\sigma_1\sigma_2$ have different geometrical shapes and so $\sigma_1\sigma_2\alpha_1\neq \alpha_2\sigma_1\sigma_2$, which is a contradiction. Thus, $w$ is not trivial word.

Now, direct computations show that $\rho(w)=\sigma_1\alpha_1\alpha_2\alpha_1\sigma_1^{-1}\alpha_2\alpha_1\sigma_1^{-1}\alpha_2\sigma_1\alpha_1\alpha_2)=I_6$. This implies that $w\in \ker(\rho)$ and so $\rho$ is unfaithful.
\end{proof}


\end{document}